\documentclass{amsart}
\usepackage{amssymb}
\usepackage{amsmath}
\usepackage{hyperref}
\usepackage{mathrsfs}
\usepackage{color}
\usepackage{graphics}
\usepackage{graphicx}
\usepackage{marvosym}
\usepackage{verbatim}
\usepackage{enumitem}
\newcommand{\what}{\widehat}%
\newcommand{\R}{\mathbb R}%
\newcommand{\C}{\mathbb C}%
\newcommand{\Z}{\mathbb Z}%
\newcommand{\N}{\mathbb N}%
\newcommand{\Sph}{\mathbb S}%
\newcommand{\B}{\mathbb B}%
\newcommand{\hc}{\mathrm c}
\newcommand{\A}{\mathbb A}

\newtheorem{theorem}{Theorem}[section]

\newtheorem{lemma}[theorem]{Lemma}
\newtheorem{proposition}[theorem]{Proposition}

\theoremstyle{definition}

\theoremstyle{definition}
\newtheorem{remark}[theorem]{Remark}
\newtheorem*{remark*}{Remark}

\numberwithin{equation}{subsection}
\numberwithin{theorem}{subsection}
\begin{document}
\baselineskip18pt
\author[M. Naik]{Muna Naik}
\address[M.  Naik]{
Harish-Chandra Research Institute, Chhatnag Road, Jhunsi, Prayagraj 211 019,
India }
\email{mnaik41@gmail.com, munanaik@hri.res.in}

\author[R. P. Sarkar]{Rudra P. Sarkar}
\address[R. P. Sarkar]{Stat-Math Unit, Indian Statistical
Institute, 203 B. T. Rd., Calcutta 700108, India}\email{rudra.sarkar@gmail.com, rudra@isical.ac.in}
\title[Asymptotic mean value property]{Asymptotic mean value property for eigenfunctions of the Laplace--Beltrami operator on  Damek--Ricci spaces}

\begin{abstract}
Let $S$ be a Damek--Ricci space equipped with the Laplace--Beltrami operator $\Delta$. In this paper we characterize all eigenfunctions  of  $\Delta $  through sphere, ball and shell averages as the radius (of sphere, ball or shell) tends to infinity. 
\end{abstract}
\maketitle

\section{introduction}
Let $S$ be a  Damek--Ricci space,  equipped with a distance $d$ and the Laplace--Beltrami operator $\Delta$ induced by its Riemannian structure. We recall that these are nonsymmetric generalizations of rank one Riemannain symmetric spaces  of noncompact type which are also solvable Lie groups.  They appeared as counterexamples to the Lichnerowicz conjecture in the noncompact case . The rank one Riemannian symmetric spaces of noncompact type form a thin subclass inside the set of Damek--Ricci spaces \cite{ADY, Damek--Ricci}. We fix the identity element $e$ of the group $S$ as the base point. We call a function $f$ on $S$ to be radial if the value of $f$ at $x \in S$ depends only on $|x|:=d(e, x)$. Thus a radial function descends naturally to a function on the nonnegative real numbers. We shall  often regard a radial function $f$ on $S$ as a function on the nonnegative real numbers, as $f(x)=f(d(e, x))$. For $\lambda\in \C$, the elementary spherical function $\varphi_\lambda$ is the unique smooth radial eigenfunction of $\Delta$ with eigenvalue $-(\lambda^2+\rho^2)$ satisfying $\varphi_\lambda(e)=1$ where $\rho$ is half of the  limit of the mean curvature of geodesic sphere as radius of the sphere tends to $\infty$. If $S$ is a rank one symmetric space, then $\rho$ coincides with  the {\em half sum of positive roots} counted with multiplicities. It follows from the definition  that $\varphi_\lambda=\varphi_{-\lambda}$ and $\varphi_{i\rho}\equiv 1$.

Purpose of this article is to establish  a characterization of eigenfunctions of the Laplace--Beltrami operator $\Delta$ through the asymptotic behaviour of the radial averages of a continuous function on $S$. We shall consider three  primary  radial averages, namely the sphere, the ball and the annular averages. To state our results and for further discussions   we need to establish a few notation. 
    
Let $\Sph(x,r)$ and $\B(x, r)$ denote respectively  the geodesic sphere and ball of radius $r > 0$ centered at $x \in S$. For $0< r<r'$, let $\A_{r,r'}(x)$ denote the annulus or shell centered at $x \in S$ with inner radius $r$ and outer radius $r'$. The volume of $\B(e, r)$ and $\A_{r,r'}(e)$ are denoted respectively by $V_r$ and $V_{r,r'}$. Let $\sigma_r$ be the normalized surface measure of $\Sph(e, r)$. For convenience we shall also use the  notation 
\[m_r=  \frac{\chi_{\B(e, r)}}{V_r} \text{ and } a_{r, r'}=\frac{\chi_{A_{r,r'}}}{V_{r,r'}}\] where $\chi_A$ is the indicator function of a set $A$.    
Using these notation we write the sphere, ball and annular averages of a  continuous function $f$ on $S$ respectively as  
\[\mathscr M_r f(x) := f \ast \sigma_r(x)= \int_{\Sph(e, r)} f(xy) \,d\sigma_r(y),\] 
\[\mathscr B_r f(x) := f \ast m_r(x)= \frac 1{V_r}\int_{\B(x, r)} f,\]
\[\mathscr A_{r, r'} f(x) := f \ast a_{r, r'}(x)= \frac 1{V_{r, r'}}\int_{\A_{r, r'}(x)} f,\]
where $\ast$ denotes the convolution of the group $S$.

The generalized (sphere) mean value property  states (\cite{Helgason-2, Helgason-3}) that a continuous function $f$ on $S$ is an eigenfunction   of $\Delta$ with eigenvalue $-(\lambda^2+\rho^2)$ for some $\lambda\in \C$, if and only if
\begin{equation} \label{MVP-sphere} f\ast \sigma_r =  \varphi_\lambda(r) f \text{ for all }r>0.
\end{equation} (See also Proposition \ref{character-eigen}.) Above, $\varphi_\lambda(x)$, being a radial function on $S$, is interpreted as a function on nonegative real numbers, as mentioned above. Therefore such an eigenfunction $f$ satisfies the ball mean value property:
\begin{equation} \label{MVP-ball}f\ast \chi_{\B(e, r)} = \left(\int_{\B(e, r)}\varphi_\lambda(x) \,dx \right)\, f, \text{ for all } r>0.
\end{equation}
From \eqref{MVP-ball} it also follows that $f$ satisfies the annular mean value property
\begin{equation} \label{MVP-annulus}f\ast \chi_{\A_{r,r'}(e)} = \left(\int_{\A_{r,r'}(e)}\varphi_\lambda(x) \,dx \right)\, f, \text{ for all } 0<r<r'.
\end{equation}
Taking $\lambda=i\rho$, we get back the standard mean value properties satisfied by the harmonic functions in all the three cases above. To simplify the statements of our main results we further introduce the following notation.

For $\lambda\in \C$ and $r'>r>0$, let
\[\sigma_r^\lambda:= \varphi_\lambda(r)^{-1}\sigma_r,\]
 \[ V_r^\lambda:=\int_{\B(e, r)} \varphi_\lambda(x)\,dx, \ \ m_r^\lambda:=(V_r^\lambda)^{-1} \chi_{\B(e, r)},\]
 \[ V^\lambda_{r,r'}:=  \int_{\A_{r,r'}(e)} \varphi_\lambda(x ) \ dx  =V^\lambda_{r'}-V^\lambda_{r} \text{ and } a^\lambda_{r, r'}= (V_{r, r'}^\lambda)^{-1} \chi_{\A_{r, r'}(e)}.\]
 In these notation  \eqref{MVP-sphere},  \eqref{MVP-ball} and 
\eqref{MVP-annulus} can be rewritten as 
 \[f\ast \sigma^\lambda_r(x)=f(x),  f\ast m^\lambda_r(x)= f(x) \text{ and } f\ast a^\lambda_{r, r'}(x)= f(x),\] whenever
 $\varphi_\lambda(r)\neq 0$ respectively   $V_r^\lambda\neq 0$, $V_{r, r'}^\lambda \ne 0$. 

The three results which we intend to prove in this paper are the following.
 \begin{theorem}
 \label{weth5} Let  $f$ and $g$ be two  continuous functions  on $S$.
If for a fixed $\lambda\in \C$, 
\[ f \ast \sigma_r^\lambda(x) \to g(x) \] for every $x\in S$,
uniformly on compact sets  as $r \to \infty$ through $\{r >0\mid \varphi_\lambda(r) \ne 0 \}$, then $\Delta g=-(\lambda^2+\rho^2) g$.
\end{theorem}

\begin{theorem} \label{weth1}   Let  $f$ and $g$ be two  continuous functions  on $S$. If for a fixed $\lambda\in \C$,
\[f\ast m_r^\lambda(x) \to g(x)\]  for every $x\in S$, uniformly on compact sets, as $r\to \infty$ through $\{r>0 \mid V_r^\lambda \ne 0\}$,  then $\Delta g=-(\lambda^2+\rho^2)g$.
 \end{theorem}

\begin{theorem}
 \label{annth1}
Let  $f$ and $g$ be two  continuous functions  on $S$. If for a fixed $\lambda\in \C$  and two fixed positive numbers $d,\delta$, \[ f \ast a_{r,r'}^\lambda(x)  \to g(x) \] for every $x\in S$, uniformly on compact sets as $(r', r) \to \infty$ through  
\[\{(r',r)\in \R^+ \times \R^+ \mid d < r'-r < d+\delta, V_{r,r'}^\lambda\neq 0\},\]
then $\Delta g = -(\lambda^2+\rho^2)g$.
\end{theorem}

We recall that  for $\lambda \in i \R$,  $\varphi_\lambda$ is a strictly positive function on $S$.   Hence for such $\lambda$, both $V_r^\lambda$ and $V_{r,r'}^\lambda$ are positive quantities.  For  $\lambda\in \C$ with $\Im \lambda \neq 0$, $\varphi_\lambda(r)$, $V^\lambda_r$ and  $V_{r, r'}^\lambda$ are nonzero  for all sufficiently large $r$ (see Section \ref{prelim} for details). As we are interested in the limit as $r\to \infty$ and $r'\to \infty$, the sets 
$\{r\mid \varphi_\lambda(r)=0\}, \{r\mid V_r^\lambda = 0\},\{(r', r)\mid V_{r,r'}^\lambda=0\}$ for such $\lambda$ do not concern us. The situation however is more delicate when $\lambda$ is a nonzero real number. Since the functions $r\mapsto \varphi_\lambda(r)$ and $r\mapsto V^\lambda_r$ are real analytic, the sets \[\{r >0 \mid \varphi_\lambda(r) =  0\} \text { and }\{r > 0 \mid V_r^\lambda = 0 \}\] are countable. 
  Similarly, the function
$(r', r)\mapsto V^\lambda_{r, r'}$ is also real analytic on the strip  
$d<r'-r<d+\delta$ on the first quadrant of the $(x, y)$-plane, hence their zeros will form a lower dimensional  subset. But for nonzero real $\lambda$, these three sets of zeros are unbounded. See Remark \ref{remark-zeros} for details.       Thus the radii  need to approach infinity avoiding these sets and that adds additional difficulties to the proof.

\subsection{Background and motivation} 
The characterization of harmonic functions through asymptotic behaviour of sphere or ball averages of a function  as the radius goes to zero is classical (see \cite{Blas1, Blas2, Bre}), and extendeds to all Riemannian manifolds (see \cite[Theorem 6.11.1]{Will}). On the other hand,  the asymptotic behaviour of these averages as the radius tends to infinity, does not seem to be very well known.

A text book proof of the fact that bounded harmonic functions on $\R^n$ are constant is the following, where in $\R^n$ we reuse the notation developed above for $S$. We have for $f\in L^\infty(\R^n)$,

\[|\mathscr B_rf(0)-\mathscr B_rf(x)| \le C \|f\|_\infty |\A_{r-|x|, r+|x|}|/V_r \to 0, \text{ as } r\to \infty.\] We observe that this proof actually shows that if for  $f\in L^\infty(\R^n)$ and a measurable function $g$ on $\R^n$, $\mathscr B_rf(x)\to g(x)$ as $r\to \infty$, pointwise for almost every  $x\in \R^n$, then  $g$ is constant, hence harmonic. 
We note the following points about this proof.  Firstly the proof crucially depends on the polynomial growth of the ball in $\R^n$, which precludes its extension to $S$. Secondly, as the proof is designed  to show that the limit function $g$ is constant, it is not robust enough to generalize to a version for eigenfunctions with arbitrary eigenvalues, even in $\R^n$. On the other hand it reminds us that in $S$ there are nonconstant harmonic functions which are bounded, a sharp distinction with the Euclidean spaces. Lastly, a growth condition such as `$f$ is bounded' is necessary for this proof.
In search of a result which considers functions without any growth condition we find a paper by  Plancherel and P\'olya (1931) \cite{PP} on $\R^2$. Among other things, this was generalized to $\R^n$  by Benyamini and Weit (1989) \cite{BW1}. We state here a simplified and combined version of this result:
\begin{theorem} \label{PP-BW}
If for  continuous functions $f, g$ on $\R^n$
\[\lim_{r\to \infty} \mathscr B_r f= g\]  uniformly on compact sets, then $g$ is a  harmonic function.
\end{theorem} 
See also \cite{Weit} where limiting behaviour of sphere average of continuous functions on $\R^2$ is considered. 
An analogue of Theorem \ref{PP-BW} (and its generalization for eigenfunctions of the Laplace--Beltrami operator), considering the ball averages, in  rank one symmetric spaces is proved recently \cite{NRS1} by the authors and their collaborator. We are not aware of any other result in this direction and none except \cite{NRS1} considered eigenfunctions.

\subsection{Organization of the paper} We define all notation and collect the required  preliminaries in Section \ref{prelim}. Next three sections contain the proofs of the three theorems, stated above.   While the first and the third theorem will be proved in details, the proof of the second, namely the one involving ball-averages, will be a sketch since with some effort, a reader will be able to construct the proof from that of the third theorem.     In the last section we illustrate the asymptotic behaviour of some continuous functions for which the computation is easy. We also provide some counter examples to justify the formulations of our results and discuss an open question.

\section{Preliminaries}\label{prelim} In this section we shall establish notation and garner all the ingredients required for this paper.
\subsection{Basic Notation}  The letters  $\R,\, \R^\times,\, \R^+, \,\C$ and $\N$ denote respectively  the set of  real numbers, nonzero real numbers, positive real numbers, complex numbers and natural numbers. For $z\in \C$,  $\Re z$ and $\Im z$  denote respectively the real and imaginary parts of $z$.
We shall follow the practice of using the letters $C, C_1, C_2, C', c$
etc. for positive constants, whose value may change from one line to another.
The constants may be suffixed to show their
dependencies on important parameters.
Everywhere in this article the symbol $f_1\asymp f_2$ for two positive expressions $f_1$ and $f_2$
means that there are positive constants $C_1, C_2$ such that $C_1f_1\leq f_2\leq C_2f_1$.
For  a set $A$ in a topological space $\overline{A}$ is  its closure and for a set $A$ in a measure space $|A|$ denotes its measure. For two  functions $f_1, f_2$,  the notation $\langle f_1, f_2\rangle$   means  $\int f_1 f_2$ if  the integral makes sense.
\subsection{Damek--Ricci space}
\label{Drspace}
Let $\mathfrak n=\mathfrak v\oplus \mathfrak z$ be a $H$-type
Lie algebra where $\mathfrak v$ and $\mathfrak z$ are vector spaces
over $\R$ of dimensions $m$ and $k$ respectively. Indeed
$\mathfrak z$ is the centre of $\mathfrak n$ and $\mathfrak v$ is
its ortho-complement with respect to the inner product of
$\mathfrak n$. Then we  know that $m$ is even.  The group law of $N=\exp \mathfrak n$ is given by
\[(X, Y). (X', Y')=(X+X', Y+Y'+\frac 12[X, X']),\ \ X\in \mathfrak v, Y\in \mathfrak z.\]
We shall identify $\mathfrak v$, $\mathfrak z$ and $N$ with $\R^m$, $\R^l$ and
$\R^m\times \R^l$ respectively.  The group  $A=\{a_t=e^t \mid t\in \R\}$  acts on $N$ by nonisotropic
dilation: $\delta_{t}(X,Y)=(e^{t/2}X, e^{t}Y)$.
Let $S=NA=\{(X, Y, a_t)\mid (X, Y)\in N, t\in \R\}$ be
the semidirect product of $N$ and $A$ under the action above. The group law of $S$ is thus given by:
\[(X, Y, a_t) (X', Y', a_s)=(X+a_{t/2} X', Y+ a_{t} Y'+ \frac  {a_{t/2}}2 [X, X'], a_{t+s}).\] It then follows that
$\delta_{t}(X,Y)=a_tna_{-t}$, where $n=(X, Y)$.
The Lie group
$S$ is  solvable, connected and simply connected  with
Lie algebra $\mathfrak s=\mathfrak v\oplus\mathfrak z\oplus\R$. It
is well known that $S$ is  nonunimodular. The
homogeneous dimension of $S$ is $Q=m/2+k$. For convenience we shall also use the notation $\rho=Q/2$.
The group $S$ is equipped with the  distance $d$ induced by the left-invariant Riemannian
metric,
\begin{equation*}\langle(X, Z, \ell), (X', Z', \ell')\rangle =\langle X, X'\rangle+\langle Z, Z'\rangle+\ell\ell'
\end{equation*} on $\mathfrak s$.
The associated left invariant Haar measure $dx$ on $S$   is given
by
\begin{equation} \int_S f(x)\,dx=\int_{N\times A}f(na_t)e^{-Q
t}\,dt\,dn,
\label{measure-NA}
\end{equation} where $dn(X,Y)=dX\,dY$ and $dX, dY, dt$ are Lebesgue
measures on $\mathfrak v$, $\mathfrak z$ and $\R$ respectively.

The group $S$ can also be  realized as the unit ball
\begin{equation*}\B(\mathfrak  s)=\{(X, Z, \ell)\in \mathfrak  s\mid |X|^2+|Z|^2+\ell^2<1\}\end{equation*}
via a Cayley transform $C:S\longrightarrow  \B(\mathfrak s)$ (see
\cite[p.~646--647]{ADY} for details). For an element $x\in S$, let
\[|x|=d(x, e)=\log \frac{1+\|C(x)\|} {1-\|C(x)\|}.\] Then  $d(a_t, e)=|t|$. 
 We identify the boundary 
$\partial \B(\mathfrak s)$ of $\mathfrak s$ with $\Sph^{n-1}$ and through this identification we express an element $x \in S$ uniquely by \[x = \exp (r w),\]
where $r = d(e, x)$, $w \in \Sph^{n-1}$ and $\exp : \mathfrak s \to S$ is the exponential map.
The left Haar measure in geodesic polar
coordinates is  given by (\cite[(1.16)]{ADY})
\begin{equation}
dx=C\left(\sinh \frac r 2\right)^{m+k}\left(\sinh  \frac r 2\right)^k\,dr\, dw 
\label{DR-polar}
\end{equation} where $r=|x|$, $dw$ denotes the normalized
surface measure on the unit  sphere $\partial \B(\mathfrak s)$ in
$\mathfrak  s$, $dr$ is the Lebesgue measure on $\R$, $n= \dim  S= m+k+1$ and $C$ is a constant depending on $n$. For convenience we write it as $dx=J(r) dr dw$ and thus the  integral formula in this cordinate reads as
\[\int_Sf(x)\,dx=\int_0^\infty \int_{\Sph^{n-1}}f(\exp(rw)) J(r) \,dr\, dw.\]
A function $f$ on $S$ is called {\em radial} if for all $x,y\in
S$, $f(x)=f(y)$ if $d(x,e)=d(y,e)$. By abuse of notation we shall sometimes consider a radial function $f$ as a function of $|x|$ and for such a function
\begin{equation}\label{polar}
 \int_Sf(x)\,dx=\int_0^\infty f(r)J(r)\,dr.
\end{equation}
Since  $\cosh t \asymp e^t$ and $\sinh t\asymp te^t/(1+t)$ for $t\ge 0$, it follows from  \eqref{DR-polar} and \eqref{polar} that for a radial function $f\in L^1(S)$,
\begin{eqnarray}
\int_S|f(x)|\,dx\asymp C_1\int_0^1|f(t)|
t^{n-1}\,dt +C_2 \int_1^\infty
|f(t)|e^{2\rho t}\,dt. \label{polar-2}
\end{eqnarray} 
For a suitable function $f$ on $S$ its radialization $Rf$ is
defined as
\begin{equation}Rf(x)=\int_{\Sph(e, \nu)}f(y)\,d\sigma_\nu(y),
\label{radialization}
\end{equation} where $\nu=|x|$ and $d\sigma_\nu$ is the normalized
surface measure induced by the left invariant Riemannian metric on
the geodesic sphere $\Sph(e, \nu)= \{y\in S\mid d(y, e)=\nu\}$. 
It is clear that $Rf$ is a
radial function and if $f$ is radial then $Rf=f$.
The following properties of the radialization operator will be needed (see \cite{DR, ACB}):
\begin{enumerate}
\item $\langle R \phi, \psi\rangle=\langle \phi, R\psi\rangle,\ \ \phi, \psi\in C^\infty_c(S)$.
\item $R(\Delta f)=\Delta(R f)$.
\end{enumerate}

To proceed towards the Fourier transform we need to introduce the notion of Poisson kernel.
The Poisson kernel $\wp:S\times N\longrightarrow \R$ is given by
$\wp(x,n)=\wp(n_1a_t, n)=P_{a_t}(n^{-1}n_1)$ where
\begin{equation}P_{a_t}(n)=P_{a_t}(X, Y)=C a_t^{Q}\left(\left(a_t+\frac{|X|^2}4\right)^2+|Y|^2\right)^{-Q},\,\, n=(X,
Y)\in N. \label{poisson}\end{equation}  The  value of $C$ is
adjusted so that \[
\int_NP_a(n)\, dn=1 \,\,\,(\text{see \cite[(2.6)]{ACB}}
). \] For $\lambda\in \C$, we define
\[\wp_\lambda(x, n)=\wp(x,n)^{1/2-i\lambda/Q}=\wp(x,n)^{-(i\lambda-\rho)/Q}.\] Then it is known that for each fixed $n\in N$,
\[\Delta \wp_\lambda(x, n)=-(\lambda^2+\rho^2) \wp_\lambda(x, n).\]
The elementary spherical function $\varphi_\lambda(x)$ is  given by
\[\varphi_\lambda(x)=
\int_N\wp_\lambda(x,n)\wp_{-\lambda}(e,n)\,dn.\]
The elementary spherical function $\varphi_\lambda(x)$ has the following
properties.
\begin{enumerate}
 \item [(i)]  $\varphi_\lambda$ is a smooth radial function on $S$,  $\Delta \varphi_\lambda= -(\lambda^2+\rho^2)\varphi_\lambda$,   $\varphi_\lambda(e)=1$.
\item[(ii)] $ \varphi_\lambda$ is a strictly positive function
for $\lambda \in i\R$ and $|\varphi_\lambda|\le \varphi_{i \Im \lambda }$ for all $\lambda\in \C$.
\item[(iii)]  $\varphi_\lambda(x)= \varphi_{-\lambda}(x)= \varphi_\lambda(x^{-1})$ for all $\lambda\in \C$, and for all $x\in S$.
  \end{enumerate}
 Moreover property (i) characterizes $\varphi_\lambda$ completely i.e.  $\varphi_\lambda$ is the unique radial smooth eigenfunction of  $\Delta$  with eigenvalue $-(\lambda^2+\rho^2)$ satisfying  $\varphi_\lambda(e)=1$. Since $\wp_{-i\rho}(x,n)\equiv 1$ for all $x\in S, n\in N$  and
$\wp_{i\rho}(x,n)=\wp(x,n)$, it follows that
\[\varphi_{-i\rho}(x)=\int_N\wp_{i\rho}(e,n)\, dn=\int_NP_1(n)\,dn=1,\]

 For $\Im \lambda <0$ and $t>0$, we have the following asymptotic estimate of $\varphi_\lambda $,
\begin{equation}
\label{esti}
 \displaystyle \lim_{t \rightarrow \infty} e^{ -(i\lambda - \rho)t} \varphi_ \lambda(t) = \hc(\lambda)
\end{equation}
 where $\hc(\lambda)$ is the analogue of Harish-Chandra $\hc$-function (see \cite[(2.7), p. 648]{ADY}).

Since  the $\hc$-function has neither zero nor pole in the region $\Im \lambda <0$ and $\varphi_\lambda=\varphi_{-\lambda}$, from this we conclude that
for any $\lambda\in \C$ with $\Im \lambda\neq 0$, there is a $t_\lambda >0$ such that
\begin{equation}
\label{varphiesti1}
 |\varphi_\lambda(t)| \asymp \varphi_{i\Im \lambda}(t) \asymp   e^{(|\Im \lambda| - \rho)t} \text{ for } t> t_\lambda.
\end{equation}
We note that $\varphi_\lambda$ is a strictly positive function for $\lambda \in i\R$. Therefore by \eqref{esti} 
\begin{equation}
\label{phiesti}
 \varphi_{\lambda}(t) \asymp e^{(|\Im \lambda|-\rho)t }\, \, \text{ for } \lambda \in i \R, \lambda \neq 0.
\end{equation}

For $\lambda=0$ we also have (see \cite{ADY})
\begin{equation}
\label{phioesti}
\varphi_0(t) \asymp (1+t)e^{-\rho t}.
\end{equation}
For a measurable function $f$ on $S$  and $\lambda \in \C$, we define the spherical Fourier transform  of $f$ at $\lambda$ by
\begin{equation} \label{pg9}
 \what{f}(\lambda) := \int_X f(x) \varphi_\lambda(x) \ dx,
\end{equation}
whenever the integral makes sense. The notation $\Sph(x,r), \B(x, r)$, $\A_{r_1, r_2}(x)$, $V_r^\lambda$ and $V^\lambda_{r_1,r_2}$ are as defined in the introduction. Thus $V_r^\lambda$ and $V^\lambda_{r_1,r_2}$  are the spherical Fourier transform at $\lambda$, of the indicator function of the ball $\B(e, r)$ and the annulus $\A_{r_1, r_2}(e)$  respectively.  Since $\varphi_\lambda= \varphi_{-\lambda}$, it is also clear that $V_r^\lambda$ = $V_r^{-\lambda}$ and $V^\lambda_{r_1,r_2}$= $V^{-\lambda}_{r_1,r_2}$.

\subsection{Jacobi functions} \label{Jacobi}
For $\alpha, \beta >-1/2$, $\lambda\in \C$ and $t\ge 0$, let $\phi_\lambda^{(\alpha, \beta)}(t)$ denote the Jacobi functions
defined by\[\phi_\lambda^{(\alpha, \beta)}(t):= \, _2F_1\bigg(\frac 1 2(\alpha+\beta+1-i\lambda), \frac 1 2(\alpha+\beta+1+i\lambda); \alpha+1;-\sinh^2(t)\bigg), \]
 where   $_2F_1(a,b;c;z)$ is the Gaussian hypergeometric function. It follows from the property of $_2F_1$ that $\phi_\lambda^{(\alpha, \beta)}=\phi_{-\lambda}^{(\alpha, \beta)}$ for $\alpha, \beta$ and $\lambda$ as above. For a detailed account on Jacobi functions  we refer to  \cite{Koorn}.
 We note here that both the symbols $\varphi_\lambda$ for elementary spherical functions on $S$ and $\phi_\lambda^{(\alpha, \beta)}$ for Jacobi functions are standard and widely used in the literature. We hope the use of these symbols will not confuse the readers.

 We recall that for specific parameters $\alpha, \beta$, Jacobi functions coincide with the elementary spherical functions, realized as  functions on nonnegative real numbers through the polar decomposition of $S$. In our parametrization, they are related in the following way (\cite{ADY}, p. 650):
\begin{equation}
\label{Jacobi-esf}
\varphi_\lambda(t)=\phi_{2\lambda}^{(\alpha, \beta)}(t/2), \text { where } \alpha=(m+k-1)/2, \beta=(k-1)/2.
\end{equation}
 However, for arbitrary $\alpha, \beta>-1/2$, a Jacobi function may not be an elementary spherical function of any Damak--Ricci space $S$. The asymptotic estimates of the elementary spherical functions described in \eqref{esti} -- \eqref{phiesti} generalizes for the Jacobi functions in the following way. 

For $\Im \lambda <0$ we have (\cite[2.19]{Koorn})
\begin{equation}
 \label{baaaeq1}
  \phi_\lambda^{(\alpha,\beta)}(t) =  \hc_{\alpha,\beta}(\lambda)e^{(i \lambda -\alpha-\beta-1)t}(1+o(1)) \text{ as } t \rightarrow \infty,
 \end{equation} where  $\hc_{\alpha,\beta}(\lambda)$  is an analogue of the Harish-Chandra $\hc$-function which has neither zero nor pole in the region  $\Im \lambda<0$.

Hence for $\lambda \in \C$ with $\Im \lambda < 0$, 
\begin{equation}
 \label{baeq1}
  \lim_{t\to \infty }e^{-(i \lambda -\alpha-\beta-1)t}\phi_\lambda^{(\alpha,\beta)}(t) =  \hc_{\alpha,\beta}(\lambda) .\end{equation}
That is  for $\lambda \in \C$ with $\Im \lambda <0$ and sufficiently small $\epsilon >0$, there exists $t(\lambda, \epsilon) >0$ such that
\begin{equation}
\label{baleq4}
  (|\hc_{\alpha, \beta}(\lambda)|-\epsilon)  e^{(|\Im \lambda |-\alpha-\beta-1)t} \le |\phi_\lambda^{(\alpha,\beta)}(t)| \le  |(\hc_{\alpha, \beta}(\lambda)|-\epsilon )  e^{(|\Im \lambda |-\alpha-\beta-1)t},
   \end{equation}
for all $t >t(\lambda, \epsilon)$.
Since $\phi_\lambda^{(\alpha,\beta)}= \phi_{-\lambda}^{(\alpha,\beta)}$,  we have for all 
$\lambda \in \C$ with $\Im \lambda \neq 0$,
\begin{equation}
\label{baeq4}
  |\phi_\lambda^{(\alpha,\beta)}(t)| \asymp e^{(|\Im \lambda |-\alpha-\beta-1)t} \text{ as } t \rightarrow \infty.
   \end{equation}

We now quote the following result from \cite[Lemma 5.2(a)]{PS},  which shows that $V_r^\lambda$  can be expressed in terms of the Jacobi functions. We recall that $V_r^\lambda$ is the Fourier transform of  $\chi_{\B(e, r)}$ at $\lambda$.
\begin{theorem}\label{Thm-Jacobi-1}
 Let $\alpha'=\frac{m+k+1}{2},  \beta'=\frac{k+1}{2}$ and $n=m+k+1=\dim S$. Then for $\lambda \in \C$ and $r>0$,

 \begin{eqnarray}
 \label{baeq2}
  V_r^\lambda &=& \frac{2^n\pi^{\frac n 2}}{\Gamma(\frac n 2+1)} \sinh^n \left(\frac {r}{ 2}\right) \cosh^{k+1} \left(\frac r 2\right)
 \phi_{2\lambda}^{(\alpha', \beta')}\left(\frac r 2\right)\\&=&\frac{4^{\alpha'}\pi^{\alpha'}}{\Gamma(\alpha'+1)} \sinh^{2\alpha'} \left(\frac {r}{ 2}\right) \cosh^{2\beta'} \left(\frac r 2\right)
 \phi_{2\lambda}^{(\alpha', \beta')}\left(\frac r 2\right). \nonumber
 \end{eqnarray}
\end{theorem}
A minor error in the expression of $V_r^\lambda$ given in \cite[Lemma 5.2(a)]{PS} is rectified here. In the first line of \eqref{baeq2} $\cosh^{k-1}(r/2)$ is corrected to $\cosh^{k+1}(r/2)$.

We end this subsection with the following estimate of the Jacobi function which we shall use.
\begin{proposition} \label{lem-ww} Let $\lambda\in \R^\times$ and $\alpha, \beta>-1/2$  be fixed.
There exists constants $A_\lambda>0$, $C_\lambda >0$, $\theta_\lambda\in \R$ and a function $\epsilon_\lambda^\ast:\R^+ \to \R$ satisfying $|\epsilon_\lambda^\ast(t)| \le C_\lambda e^{-2t}$, depending only on $\alpha, \beta$ and $\lambda$ such that for $t>0$,
\[(\sinh t)^{\alpha+1/2} (\cosh t)^{\beta+1/2} \phi_\lambda^{(\alpha, \beta)}(t) =  A_\lambda[\cos(\lambda t+\theta_\lambda) + \epsilon_\lambda^\ast(t)].\]
\end{proposition}
For a proof of the proposition above we refer to  \cite[(6.12)--(6.15)]{WW}.

\begin{remark}\label{remark-zeros}
Let us restrict to $\lambda\in \R^\times$. It is clear from the proposition above and \eqref{Jacobi-esf} that $\varphi_\lambda(r)$ and hence $V^\lambda_r$ and $V^\lambda_{r,r'}$ are real numbers for $r>0, r'>0$. It also follows from the proposition, \eqref{Jacobi-esf} and Theorem \ref{Thm-Jacobi-1} that the sets of zeros of the functions $r\mapsto \varphi_\lambda(r)$ and $r\mapsto V^\lambda_r$ are sequences diverging to infinity. It can also be shown that except possibly some carefully chosen $d$ and $\delta$, the set $\{(r',r)\mid d<r'-r<d+\delta, V^\lambda_{r,r'}=0\}$ is also an unbounded set. This was alluded to in the introduction.
 
\end{remark}


\subsection{Geodesic convexity of the distance function}
Riemannian manifolds of non-positive curvatures, hence in particular the  Damek--Ricci spaces are $\mathbf{\mathrm CAT}(0)$ spaces. We need the following property of the distance function in  $\mathbf{\mathrm CAT}(0)$ spaces (see \cite[p. 176, Prop 2.2, Chap II]{Brid-Haef}, \cite[p. 24, Chap 1, Prop 5.4]{Ballman}).
\begin{proposition} \label{convex-distance}
Let  $\mathcal M$ be a $\mathbf{\mathrm CAT}(0)$ space and $x_0\in \mathcal M$. Then the distance function $x\mapsto d(x_0, x)$ from   $\mathcal M\to \R$ is geodesically convex, i.e. given any  geodesics  $\gamma: [0, 1]\to \mathcal M$, parameterized proportional to arc length, the following inequality holds for all $t\in [0, 1]:$
\[d(x_0, \gamma(t))\le (1-t) d(x_0, \gamma(0)) + t d(x_0, \gamma(1)).\]
\end{proposition}

Using this property we shall prove the following important step towards Theorem \ref{annth1}.
\begin{proposition} \label{lemma-convex} Let $x_o\in S$, $w \in \Sph^{n-1}$ and $r> d(x_o, e)$ be fixed. Suppose that for some $s_o>0$, $d(x_o, \exp (s_ow))=r$.
Then for any positive  $s$,   $d(x_o, \exp (sw))>r$ if and only if $s>s_o$.
\end{proposition}
\begin{proof}
By Proposition \ref{convex-distance}, the  function
\[\alpha(s)=d(x_o, \exp (sw)), \,\,\, s\ge 0\] is convex. If $s\in (0, s_o)$ then it follows from convexity of $\alpha$ and the assumption that $\alpha(0)< \alpha(s_o)=r$ that
\[\alpha(s)=d(x_o, \exp (sw))<r.\]
If $s_o\in (0, s)$ and $\alpha(s) \le \alpha(s_o)=r$, then again convexity of $\alpha$ implies that $\alpha(s_o)<r$,
which is a contradiction. Hence $d(x_o, \exp (sw))>r$.
\end{proof}

\subsection{Characterization of eigenfunction} \label{characterization-eigenfucntion}
We recall that the average of a function $f$ on a sphere of radius $t>0$ around $x \in S$ is given by
   \[\mathscr M_tf(x)=\int_{|y|=t} f(xy) \, d\sigma_t(y) =f\ast \sigma_t(x),\]
   where $\sigma_t$ is the normalized surface measure on the sphere of radius $t$  with center at the identity $e$.

 For a Riemannian manifold $\mathcal M$, $\dim \mathcal M=n$, $f\in C^\infty(\mathcal M)$ and a point $x \in \mathcal M$, we have (\cite[(6.185), p. 249]{Will})
  \begin{equation}\mathscr M_tf(x)=f(x)+\frac{1}{2n}\Delta f(x)t^2 +O(t^4).
  \label{mean-laplace}
  \end{equation} 
  
 Using \eqref{mean-laplace} one can prove the following.
\begin{proposition} \label{Blaschke-S} If a  function $f\in C^\infty(S)$  satisfies
 \[\lim_{t \to 0} \frac{\mathscr M_t f- \varphi_\lambda(t) f}{t^2}=0,
 \]
then $\Delta f = -(\lambda^2+\rho^2)f$.
\end{proposition}
See \cite[Proposition 2.4.4]{NRS1} for a proof of the above proposition for rank one symmetric spaces of noncompact type. The proof for Damek--Ricci spaces is exactly the same, except that it requires to substitute  the notation $m_\gamma$, $m_{2\gamma}$  and $\varphi_\lambda(a_t)$ by 
$m, k$ and $\varphi_\lambda(t)$ respectively.

This yields the following characterization of eigenfunction of $\Delta$  from the   generalized 
mean value property.

\begin{proposition} \label{character-eigen}
Let $\delta >0$  and $\lambda\in \C$. Let $f$ be a continuous function $S$ such that
 \[\mathscr M_tf(x)=  \varphi_\lambda(t)f(x),\] for  every $x\in S$ and for  every  $t$ with $0<t <\delta$, then $\Delta f=-(\lambda^2+\rho^2) f$.
\end{proposition}
\begin{proof}
We take a ball $\B(e,r)$ of radius $r<\delta$ with center $e$ and a
 radial function $h \in C_c^\infty(S)$ with its support  contained inside $\B(e,r)$ which satisfies
$\int_{\B(e,r)} \varphi_\lambda(z) h(z)\,  dz= 1$.
Then for every $x \in S$, \begin{eqnarray*}
      f\ast h(x) &=& \int_S f(xz)h(z) \, dz\\
      &=& \int_0^r \mathscr M_tf(x) h(t) J(t)\,dt\\
         & =&f(x)\int_0^r \varphi_\lambda(t)h(t)J(t)\,dt\\
      &=&f(x).
     \end{eqnarray*}
Therefore $f\in C^\infty(S)$. The assertion is now immediate from Proposition \ref{Blaschke-S}.

\end{proof}

\section{Proof of Theorem \ref{weth5}}
\label{spheremvp}
In the following lemma, we verify two equations, required for the proof of the main result.
\begin{lemma}
\noindent \label{propo1}
For any $\lambda\in \C$, $t>0, s>0$ and $f\in L^1_{loc}(S)$ we have the following equalities.
\[\mathrm{(i)} \ \ \ \  \ \int_{\Sph(e, t)}\varphi_\lambda(d(a_s,y)) \, d\sigma_t(y)= \varphi_\lambda(t)\varphi_\lambda(s).\]
\[\mathrm{(ii)} \ \ \ \ \ \ f \ast \sigma_t \ast \sigma_s(x)=  \int_{\Sph(e, s)} f \ast \sigma_{d(a_t,y)}(x) \,d\sigma_s(y).\] 
 \end{lemma} 
\begin{proof}

\noindent (i) This is clear as the left hand side is $\varphi_\lambda\ast \sigma_t$ and $\what{\sigma_t}(\lambda)=\varphi_\lambda(t)$. Indeed,  
\begin{eqnarray*}
\int_{\Sph(e, t)}\varphi_\lambda(d(a_s,y)) \, d\sigma_t(y) &=& \int_{\Sph(e, t)}\varphi_\lambda(d(e, y^{-1}a_s)) \, d\sigma_t(y)\\&=&  \int_{\Sph(e, t)}\varphi_\lambda(a_{-s}y) \, d\sigma_t(y)\\&=&
 \varphi_\lambda \ast \sigma_t(a_{-s}) \\&=& \what{\sigma_t}(\lambda)\varphi_\lambda(a_{-s})\\&=& \varphi_\lambda(t) \varphi_\lambda(s).
\end{eqnarray*}
(ii) It is enough to show that as measures,
\begin{equation}
\label{href3}
 \sigma_t\ast \sigma_s = \int_{\Sph(e, s)}
\sigma_{d(a_t, y)} \,d\sigma_s(y),
\end{equation}
which follows from the injectivity of the spherical Fourier transform.
Indeed, the spherical Fourier transform at $\lambda$ of the left hand side  of \eqref{href3} is
\[\sigma_t\ast \sigma_s(\varphi_\lambda) = \sigma_t(\sigma_s\ast \varphi_\lambda)=
\sigma_t(\varphi_\lambda(s) \varphi_\lambda)=\varphi_\lambda(s) \varphi_\lambda(t).\]
The spherical Fourier transform at $\lambda$ of the right hand side of \eqref{href3} is
\begin{align*}
&\int_{\Sph(e, s)} \sigma_{d(a_t, y)}(\varphi_\lambda) \,d\sigma_s(y)\\
=& \int_{\Sph(e, s)} \varphi_\lambda(d(a_t, y))\,d\sigma_s(y)\\
=& \varphi_\lambda(s) \varphi_\lambda(t).
\end{align*} In the last step we have used (i).
\end{proof}

\begin{proof}[Proof of Theorem \ref{weth5}] We shall deal with two disjoint sets of $\lambda$ seprately.

\noindent {\bf Case 1}. Let $\lambda \in \C\setminus  \R^\times$. We note that for $\lambda\in i\R$, $\varphi_\lambda$ is positive and for $\lambda\in \C\setminus \R$, $\varphi_\lambda(t)\neq 0$ when $t$ is large (see \eqref{varphiesti1}). As we are concerned with $t\to \infty$, the $\varphi_\lambda(t)$ in the denominator of $\sigma^\lambda_t$ poses no problem when $\lambda$ is in this set.  

We fix $x \in S$ and $s >0$. We have,
\begin{eqnarray}
\label{ref1}
\frac 1{\varphi_\lambda(t)} f\ast \sigma_t \ast \sigma_s(x) = \frac 1 {\varphi_\lambda(t)} \int_{\Sph(e,r)}f\ast \sigma_t(xy)\, d\sigma_s(y)
\end{eqnarray}
Since the geodesic sphere $\Sph(x,s)$ is compact in $S$, from hypothesis and \eqref{ref1} it follows that \begin{equation}
  \lim_{t \to \infty}\frac 1{\varphi_\lambda(t)} f\ast \sigma_t \ast \sigma_s(x) = g\ast \sigma_s(x).
   \label{href4}                                                                                                    \end{equation}

We shall show that
\begin{equation}\label{href5}\frac1{\varphi_\lambda(t)} f\ast \sigma_t\ast \sigma_s \to \varphi_\lambda(s) g(x), \ \ \ \text{ as } t\to\infty. \end{equation}
Using Lemma \ref{propo1} we have,

\begin{eqnarray} \label{href6} &&\frac1{\varphi_\lambda(t)} f\ast \sigma_t\ast \sigma_s - \varphi_\lambda(s) g(x)\\ \nonumber
&=&\frac1{\varphi_\lambda(t)} \left[ \int_{\Sph(e, s)} f\ast \sigma_{d(a_t, y)}(x) \,d\sigma_s(y) - \varphi_\lambda(t) \varphi_\lambda(s) g(x)\right] \\
&=&\frac1{\varphi_\lambda(t)} \int_{\Sph(e, s)} \varphi_\lambda(d(a_t, y)) \left[ \frac 1 {\varphi_\lambda(d(a_t, y)) } f\ast \sigma_{d(a_t, y)}(x)  -  g(x)\right] d\sigma_s(y). \nonumber
\end{eqnarray}

As $t\to \infty$, $d(a_t, y)\to \infty$ for any $y$ with $|y|=s$ as $s$ is fixed. Therefore by the hypothesis given any $\epsilon>0$, there exists $M>0$ such that for $t>M$,
\[\left| \frac 1 {\varphi_\lambda(d(a_t, y)) } f\ast \sigma_{d(a_t, y)}(x)  -  g(x)\right|< \frac \epsilon {\varphi_{i\Im \lambda}(s)},\] which is valid for all $y$ with $|y|=s$.

Hence for sufficiently large $t$  we have,
\begin{eqnarray*}
\left|\frac1{\varphi_\lambda(t)} f\ast \sigma_t\ast \sigma_s - \varphi_\lambda(s) g(x)\right|
&\le&\frac \epsilon{|\varphi_\lambda(t)|\varphi_{i\Im \lambda}(s)} \int_{|y|=s} \varphi_{i\Im\lambda}(d(a_t, y)) \, \,d\sigma_s(y)\\&=&
\frac {\varphi_{i\Im \lambda}(t)}{|\varphi_\lambda(t)|} \ \ \epsilon. 
\end{eqnarray*}
In the last step we have used Lemma \ref{propo1}(i).
For $\lambda\in i\R$,  this proves \eqref{href5}. For other $\lambda$ in this set we use \eqref{varphiesti1}, i.e. for all sufficiently large $t$,  
 \[C'\varphi_{i\Im \lambda}(t) \le |\varphi_\lambda(t)| \le C'' \varphi_{i\Im \lambda}(t),\] for some constants $C'>0, C''>0$.  Hence \eqref{href5} is proved.

From  \eqref{href4} and \eqref{href5} we conclude that
\[g \ast \sigma_s(x) = \varphi_\lambda(s) g(x).\]

Hence by Proposition \ref{character-eigen}, $\Delta g= -(\lambda^2+\rho^2)g$.
\vspace{.2in}

\noindent{\bf Case 2: }
We now take   $\lambda \in \R^\times$. We recall that for these $\lambda$, $\varphi_\lambda(t)$ can be zero on a unbounded discrete set of points.
Therefore additional care is needed to deal with  $\varphi_\lambda(t)$ which appears in the denominator of $\sigma^\lambda_t$. To circumvent this,   we  shall find a $\delta>0$ and a sequence of positive reals
$\{t_n\}_{n \in \N}$ with  the following properties:
\begin{itemize}
 \item [(a)]$\varphi_\lambda(t_n)$ is positive and for $y\in S$ with $t_n-\delta\leq |y| \leq t_n+\delta$, $\varphi_\lambda(y)>0$, 
\item[(b)]$\varphi_\lambda(z)$ is positive for $z\in S$ with $0 \leq |z|\le \delta$,
\item[(c)]$t_n \rightarrow \infty$ as $n \rightarrow \infty$.
\end{itemize}
For $\lambda \in \R^\times $ and $t \ge 1$  the Harish-Chandra
series for $\varphi_\lambda$ implies that
  \[\varphi_\lambda(t)=e^{-\rho t}[\hc(\lambda)e^{i\lambda t}+c(-\lambda)e^{-i\lambda t}+E(\lambda, t)],\] where
$|E(\lambda,t)|\le A_\lambda e^{-2t}$. For a proof of this see \cite[(3.11)]{Ion-Pois-1}. Although the paper \cite{Ion-Pois-1} is concerned with the symmetric spaces, the proof  works for the Damek--Ricci spaces, as the symmetric spaces are dealt as $N\rtimes A$ in this paper.  Let $\hc(\lambda)= a(\lambda)+ib(\lambda)$, where $a(\lambda), b(\lambda)$ are respectively the real and imaginary parts of $\hc(\lambda)$. Using the fact that $\hc(-\lambda)= \overline {\hc(\lambda)}$,
we get
\begin{align}\label{speesti1} \varphi_\lambda(t) &= e^{-\rho t}[\Re(\hc(\lambda)e^{i\lambda t})+E(\lambda,t)]\\ \nonumber
&= e^{-\rho t}[a(\lambda)\cos (\lambda t)-b(\lambda) \sin (\lambda t)+E(\lambda,t)]\\ \nonumber
&=e^{-\rho t} [C_\lambda \cos (\lambda t +\theta_\lambda)+E(\lambda, t)]
\end{align} for some constant $C_\lambda > 0$ and $\theta_\lambda \in \R$.
 Thus the zeros of $\varphi_\lambda(t)$ are  the zeros of
 \[u(t)= C_\lambda \cos(\lambda t + \theta_\lambda)+ E(\lambda, t).\]
 We find  a $t_0>0$ such that
$|E(\lambda, t)| < \frac {C_\lambda} 2$ for $t > t_0 $. Let 
\[t_n = \frac{2n\pi-\theta_{\lambda}}{\lambda} \text{ and } \delta_1 = \frac{\pi}{3\lambda}.\] If $t_n-\delta_1 \leq |y| \leq t_n+\delta_1$, then
\[2n \pi-\frac{\pi}{3} \le \lambda |y|+\theta_{\lambda} \le 2n\pi +\frac{\pi}{3} \text{ and } \frac 1 2 \le \cos (\lambda |y| +\theta_{\lambda}) \le 1.\] We take $n$ large enough so that  $t_n > t_0 +\delta_1$. Then $\varphi_\lambda(y)>0$  whenever  $t_n-\delta_1\leq |y| \leq t_n+\delta_1$.

Further as $\varphi_\lambda$ is real valued and continuous on $\R$ and  $\varphi_\lambda(e)=1$,  there exists $\delta_2 >0$
such that $\varphi_\lambda(z)$ is positive for  $0 \leq |z|\le \delta_2$. We choose $\delta = \text{min}\{\delta_1,\delta_2\}$. It is clear that we have the desired sequence $\{t_n\}_{n \in \N}$  possibly after  re-indexing and the required $\delta$.

Rest  of the proof is structurally  similar to  Case 1, although it needs some crucial modifications. We shall use the sequence $\{t_n\}$ obtained above, in the following way. 

Through the argument in Case 1,  we have,  
\begin{equation}
 \lim_{n \to \infty}\label{first2} \frac 1{\varphi_\lambda(t_n)} f \ast \sigma_{t_n} \ast \sigma_s (x) = g\ast \sigma_s(x)= \mathscr M_s g(x),
\end{equation}
for all $x \in S$, $s>0$ and $t_n$ as above.

Now we fix an  $s$ with $0<s\le \delta$ and an element $x \in S$.
From the hypothesis we have that given an $\epsilon>0$,  there exists a $r_o > 0$ such that if $t>r_o$, then
\[\left| \frac 1 {\varphi_\lambda(d(a_t, y)) } f\ast \sigma_{d(a_t, y)}(x)  -  g(x)\right|<\frac{\epsilon}{\varphi_0(s)},\] for any $y\in S$ with $|y|=s$.
Choose $N_0$ such that $t_n>r_o+\delta$ for all $n \ge N_0$.
Since $|y|=s$ and $s\le \delta$,  by triangle inequality we have \[ t_n-\delta  \le d(a_{t_n},y)\le t_n+\delta.\]
Therefore both $\varphi_\lambda(d(a_{t_n},y))$ and $\varphi_\lambda(t_n)$ are positive.
Now we can apply  the argument in Case 1 starting from  \eqref{href6}, to assert that for $n\ge N_0$,
\begin{eqnarray*}
\left|\frac 1{\varphi_\lambda({t_n})} f \ast \sigma_{t_n} \ast \sigma_s(x)-\varphi_\lambda(s) g(x)\right| &<& \frac{\epsilon}{\varphi_{0}(s) \varphi_\lambda(t_n)} \int_{|y|=s}  {\varphi_{\lambda}(d(a_{t_n},y))}\, d\sigma_s(y)
\\&<& \epsilon
\end{eqnarray*}
Hence
\begin{equation}
\label{href7}
 \lim_{n \to \infty}\frac 1{\varphi_\lambda(t_n)} f \ast \sigma_{t_n} \ast \sigma_s(x)= \varphi_\lambda(s) g(x).
\end{equation}
From \eqref{first2} and \eqref{href7} we have \[\mathscr M_s g = \varphi_\lambda(s) g,\] for all $x \in S$ and for all  $s$ with $0<s < \delta$. Hence by Proposition \ref{character-eigen} we get
$\Delta g= -(\lambda^2 +\rho^2)g.$
\end{proof}

\section{Proof of Theorem \ref{annth1}}
\label{annulusmvp}
Following two lemmas are two crucial  steps towards the proof.
\begin{lemma}
\label{alpro1}
 Fix $d >0$, $\delta >0$ and $\lambda \in \C$. Then there  exist sequences $\{r_j\}_{j \in \N}, \{r'_j\}_{j \in \N}$, both diverging to $\infty$ with the property  $d < r'_j-r_j< d+\delta$ and a number $\delta'>0$ such that  if $ t \in [r'_j-\delta', r'_j+\delta']$ and
  $s \in [r_j-\delta', r_j+\delta']$, then
  $d<t-s < d +\delta$,  $V^\lambda_{s,t} \ne 0$ and,
  \[\left|\frac {V^\lambda_{s, t}}{V^\lambda_{r_j, r'_j}}\right| \le C,\]
  for some constant $C >0 $ independent of $t, s$ and $j$.
\end{lemma}

\begin{proof}
We shall divide $\C$ in three disjoint subsets and consider $\lambda$ from them in the following three cases.

\noindent {\bf Case I: } Let $\lambda  \in i\R$. Let us take $0<\delta' < \frac{\delta}{4}$, $r_j=j$ and $r'_j=j+d+\frac{\delta}2$.

If $t  \in [r'_j-\delta', r'_j+\delta']$ and
  $s \in [r_j-\delta', r_j+\delta']$, then
\begin{equation}
\label{dseq44}
r'_j-r_j-2\delta' \le t-s \le r'_j-r_j+2\delta',
\end{equation}
hence  $d<t-s< d+\delta$. Since $\varphi_\lambda$ is a strictly positive function and $s<t$, it is clear that $V^\lambda_{s,t} >0$. It is also straightforward
from \eqref{polar-2}, \eqref{phiesti} and \eqref{phioesti} that, there exists  $j_0 \in \N$ depending on $\lambda$ such that  for all $j \ge j_0$,
\[\frac{V^\lambda_{s,t}}{V^\lambda_{r_j,r'_j}} \le \frac{ V^\lambda_{r_j-\delta/4, r'_j+\delta/4}}{V^\lambda_{r_j,r'_j}} \le C\]  for some constant $C>0$.
The desired sequences are thus $\{r_j\}, \{r_j'\}$ starting from $j_0$.
\vspace{.2in}

{\noindent \bf Case II: } Let $\lambda \notin \R \,\cup\, i\R$. In this case also we take  $0<\delta' < \frac{\delta}{4}$, $r_j=j$ and $r'_j=j+d+\frac{\delta}2$.
Let us temporarily use the notation $a$ for  $|\Im \lambda|+\rho$. Then from \eqref{baleq4} and \eqref{baeq2} it follows that for sufficiently large $r$,
\[C_1e^{ar}\le V_r^\lambda \le C_2 e^{ar},\] with $C_1, C_2$ satisfying   $C_1 e^{a(d+\frac \delta 2-2\delta')} > C_2$.

If $j$ is sufficiently large, then
\begin{equation}
\label{baleq11}
 |V^\lambda_{r_j,r'_j}| \ge |V^\lambda_{r'_j}|-|V^\lambda_{r_j}| \ge
  C_1e^{ar'_j}-C_2e^{ar_j}\ge e^{ar_j}(C_1e^{a(d+\frac \delta 2)}-C_2) >0,
\end{equation}

 and for  $t  \in [r'_j-\delta', r'_j+\delta']$,    $s \in [r_j-\delta', r_j+\delta']$, as above we have $d<t-s<d+\delta$ and
  \[|V^\lambda_{s,t}| \ge |V^\lambda_t|-|V^\lambda_s| \ge
  C_1e^{ar'_j-a\delta'}-C_2e^{ar_j+a \delta'}\ge e^{a(r_j+\delta')}(C_1e^{a(d+\frac \delta 2-2\delta')}-C_2) >0. \]
   Using \eqref{baleq11}, for large $j$ and $s, t$ as above, we get
\begin{eqnarray*}
 \left|\frac {V^\lambda_{s, t}}{V^\lambda_{r_j, r'_j}}\right| \le  \frac{|V^\lambda_s|+|V^\lambda_t|}{ |V^\lambda_{r'_j}|-|V^\lambda_{r_j}|}\le {}
 \frac{C_2e^{aj}(e^{a\delta'}+e^{a(d+\frac\delta 2+\delta')})}{e^{aj}(C_1e^{a(d+\frac{\delta} 2})-C_2)}= \frac{C_2(e^{a\delta'}+e^{a(d+\frac\delta 2+\delta')})}{C_1e^{a(d+\frac{\delta} 2)}-C_2}.
\end{eqnarray*}
Thus as  Case I, $\{r_j\}, \{r_j'\}$ starting from an adequately large $j$, are the sequences required for the assertion.
\vspace{.2in}

{\noindent \bf Case III:} Let $\lambda \in \R^\times$. Owing to \eqref{baeq2} and Proposition \ref{lem-ww} we get
\[V_r^\lambda= h(r)\left( \cos(\lambda r+\theta_{2\lambda}) + \epsilon_{2\lambda}^\ast\left(\frac {r} 2\right)\right),\]
 where  \[h(r)=\frac{4^{\alpha'}\pi^{\alpha'} A_{2\lambda}}{\Gamma(\alpha'+1)} \sinh^{\alpha'-\frac 12} \left(\frac {r}{ 2}\right) \cosh^{\beta'-\frac 12} \left(\frac r 2\right),\]
$\alpha'=(m+k+1)/2, \beta'=(k+1)/2$ and $A_{2\lambda}$ is a positive constant, $\theta_{2\lambda}\in \R$ as in Proposition \ref{lem-ww}.

This implies that
\begin{eqnarray*}
 \frac {V^\lambda_{s, t}}{V^\lambda_{r, r'}}
 &=&\frac{h(t)[\cos(\lambda t+\theta_{2\lambda}) + \epsilon_{2\lambda}^\ast(\frac {t} 2)]-h(s)[\cos(\lambda s+\theta_{2\lambda}) + \epsilon_{2\lambda}^\ast(\frac {s} 2)]}{h(r')[\cos(\lambda r'+\theta_{2\lambda}) + \epsilon_{2\lambda}^\ast(\frac {r'} 2)]-h(r)[\cos(\lambda r+\theta_{2\lambda}) + \epsilon_{2\lambda}^\ast(\frac {r} 2)]}\\
 &=&\left[\frac{h(t)}{h(r')}\right]\frac{[\cos(\lambda t+\theta_{2\lambda}) + \epsilon_{2\lambda}^\ast(\frac {t} 2)]-\frac{h(s)}{h(t)}[\cos(\lambda s+\theta_{2\lambda}) + \epsilon_{2\lambda}^\ast(\frac {s} 2)]}{[\cos(\lambda r'+\theta_{2\lambda}) + \epsilon_{2\lambda}^\ast(\frac {r'} 2)]-\frac{h(r)}{h(r')}[\cos(\lambda r+\theta_{2\lambda}) + \epsilon_{2\lambda}^\ast(\frac {r} 2)]}.
   \end{eqnarray*}

We fix $\delta_1>0$ such that $d < \delta_1< d+\delta$ and $\sin(\lambda \delta_1) \ne 0$. We choose an increasing  sequence $\{r_j \}_{j \in \N}$ of positive numbers diverging to $\infty$ with the property
 \[ \sin ({\lambda r_{j}}+\theta_{2\lambda})= 1, \,\, \cos ({\lambda r_{j}}+\theta_{2\lambda})= 0. \]
We take $r'_j = r_j +\delta_1$. Then,
    \[\cos(\lambda r_j'+\theta_{2\lambda})=\cos(\lambda r_j+\theta_{2\lambda}+\lambda\delta_1)=-\sin(\lambda\delta_1) \ne 0.\]
We choose $\delta_2 >0$ sufficiently small such that
$|\cos({\lambda u}+\theta_{2\lambda})| \ge \xi,$ for some positive real number $\xi$ whenever $ r'_j- \delta_2\le u \le r'_j+\delta_2$.
 If $s \in [r_j-\delta_2, r_j+\delta_2]$ and $t \in [r'_j-\delta_2, r'_j+\delta_2]$, then there exists positive  constants
 $D_1$ and $D_2$ such that
 \begin{eqnarray}
 \label{aleq10}
  D_1 \le \liminf_{j}\frac{|h(t)|}{|h(r'_j)|}\le \limsup_{j}\frac{|h(t)|}{|h(r'_j)|}\le D_2,\\ \nonumber
   D_1 \le \liminf_{j}\frac{|h(r_j)|}{|h(r'_j)|}\le \limsup_{j}\frac{|h(r_j)|}{|h(r'_j)|}\le D_2,\\ \nonumber
   \text{and } D_1 \le \liminf \frac{|h(s)|}{|h(t)|}\le \limsup\frac{|h(s)|}{|h(t)|}\le D_2.
       \end{eqnarray}
Now we choose $\delta_3 >0$ small enough such that if $s \in [r_j-\delta_3, r_j+\delta_3]$, then
   \[ D_2|\cos(\lambda s+\theta_{2\lambda})| \le \xi/4\]

 We fix a $\delta' >0$ such that
\begin{equation}
\label{aleq4}
 \delta' < \text{min} \left \{\frac{\delta_1-d}2, \frac{\delta+d-\delta_1}{2}, \delta_2, \delta_3 \right\}.
\end{equation}

If $s \in [r_j-\delta',r_j+\delta']$ and $t \in [r'_j-\delta',r'_j+\delta']$, then

\begin{equation}
\label{aleq1}
 \delta_1-2\delta' \le t-s\le \delta_1+2\delta'
\end{equation}

Therefore by \eqref{aleq4} and \eqref{aleq1}
\[d < t-s < d+\delta.\]
Since $\epsilon_{2\lambda}^\ast(u)=o(e^{-u}) $, we can find a $j_0\in \N$ such that if $s \in [r_j-\delta',r_j+\delta']$ and $t \in [r'_j-\delta',r'_j+\delta']$, for any  $j \ge j_o$,
\begin{equation}
 \label{aleq7}
\left|\epsilon_{2\lambda}^\ast\left(\frac {t} 2\right)-\left(\frac{h(s)}{h(t)}\right)\epsilon_{2\lambda}^\ast\left(\frac {s} 2\right)\right| \le \frac \xi 4 \text{ and } \left|\epsilon_{2\lambda}^\ast\left(\frac {r_j'} 2\right)-\left(\frac{h(r_j)}{h(r'_j)}\right) \epsilon_{2\lambda}^\ast\left(\frac {r_j} 2\right)\right| \le \frac \xi 4.
\end{equation}

Then for $j \ge j_0$ and $s,t$ as above we get
\begin{equation}
\label{aleq5}
 \frac \xi 2 \le \left|\cos(\lambda t+\theta_{2\lambda}) + \epsilon_{2\lambda}^\ast\left(\frac {t} 2\right)-\frac{h(s)}{h(t)}\left[\cos(\lambda s+\theta_{2\lambda}) + \epsilon_{2\lambda}^\ast\left(\frac {s} 2\right)\right]\right|
  \le 1+\frac \xi 2,
\end{equation}
and
\begin{equation}
\label{aleq6}
 \frac {\xi} 2\le\left|\cos(\lambda r'_j+\theta_{2\lambda}) + \epsilon_{2\lambda}^\ast\left(\frac {r'_j} 2\right)-\frac{h(r_j)}{h(r_j')}\left[\cos(\lambda r_j+\theta_{2\lambda}) + \epsilon_{2\lambda}^\ast\left(\frac {r_j} 2\right)\right]\right|\le 1+ \frac \xi 2.
\end{equation}

Hence if $s \in [r_j-\delta', r_j+\delta']$ and $t \in [r'_j-\delta', r'_j+\delta']$, owing to \eqref{aleq10}, \eqref{aleq5} and \eqref{aleq6} we get
 \[ C'\le \liminf_j \left|\frac {V^\lambda_{s, t}}{V^\lambda_{r_j, r'_j}} \right| \le \limsup_j \left|\frac {V^\lambda_{s, t}}{V^\lambda_{r_j, r'_j}}\right | \le C,\]  for some constants $C >0$ and $C'>0$. We obtain the desired sequences
 after re-indexing $\{r_j\}, \{r'_j\}$ suitably. \qedhere
\end{proof}

In the next lemma, $\lambda$, $d, \delta$ and  $a^\lambda_{r,r'}$ are as in Theorem \ref{annth1}. For convenience, we shall write $r\to \infty$ to mean $r\to \infty$ with $d<r'-r<d+\delta$.
\begin{lemma}
\label{basic-lemma-3}
 For $j \in \N$, let $\mu_j := a^\lambda_{r_j,r'_j}$ where $r_j, r'_j$ are as in Lemma \ref{alpro1} .
 Let $f$ be a radial continuous function on $S$ such that it satisfies $\displaystyle \lim_{r \to \infty} f\ast a^\lambda_{r,r'}(e)=L$.
 Then there exists a neighbourhood $N_e$ of $e$ such that $\displaystyle \lim_{j \to \infty} f\ast \mu_j(x)=L \varphi_\lambda(x)$ for any $x\in N_e$. \end{lemma}

\begin{proof} As $\lambda$ is fixed we shall write $a_{r,r'}$ for $a^\lambda_{r,r'}$ and $V_{r,r'}$ for $V_{r,r'}^\lambda$, unless it is required to mention $\lambda$.
Since by definition $\what{a_{r, r'}}(\lambda)=1$, we have
 $\varphi_\lambda\ast a_{r,r'}(x)=\varphi_\lambda(x) $  and hence in particular  $\varphi_\lambda\ast a_{r,r'}(e)=1$.

Take $h(x)=f(x)-L\varphi_\lambda(x)$.  Then
\[h\ast a_{r,r'}(e)=f\ast a_{r,r'}(e)-L \varphi_\lambda\ast a_{r,r'}(e)= f\ast a_{r,r'}(e)-L.\] Therefore by the  hypothesis  $h\ast a_{r,r'}(e)\to 0$.  Since,
 \[h\ast \mu_j(x)=f\ast \mu_j(x)-L \varphi_\lambda\ast \mu_j(x)= f\ast \mu_j(x)-L\varphi_\lambda(x),\]
we need to show that $h\ast \mu_j(x)\to 0$ as $j \to \infty$.
Thus we rewrite the statement to prove as the following:

 Let $f$ be a radial continuous function on $S$. If $\displaystyle \lim_{r\to \infty} f\ast a_{r,r'}(e)=0$, then
 there exists a neighbourhood $N_e$ of $e$ such that $\displaystyle \lim_{j \to \infty} f\ast \mu_j(x)=0$ for any $x\in N_e$.

As $f$ is radial, it follows from the polar decomposition \eqref{polar-2} that
\begin{eqnarray*}
f \ast a_{r,r'}(e) &=& \frac 1{V_{r,r'}}\int_{\Sph^{n-1}}\int_{r}^{r'}f(\exp sw)J(s)\ ds \ dw \\
&=& \frac 1{V_{r,r'}}\int_{\Sph^{n-1}}\int_{r}^{r'}f(s)J(s)\ ds \ dw \\&=&
 \frac 1{V_{r,r'}}\int_{r}^{r'}f(s)J(s)\ ds
\end{eqnarray*}

Hence from hypothesis we have
  \begin{eqnarray}
  \label{dseq5}
\displaystyle\lim_{r\rightarrow \infty} \frac 1{V_{r,r'}}\int_{r}^{r'}f(s)J(s)\ ds= 0 \text { whenever } d<r'-r<d+\delta. \label{dseq1}
 \end{eqnarray}

Fix $x \in S$ with $|x| < \delta'$ where $\delta'$ as in Lemma \ref{alpro1}. For $t \ge 0$ and $w \in \Sph^{n-1}$, we have by traingle inequality,
\begin{equation}
\label{dseq2}
  t -|x|\le |\exp(-tw)x| \le t+|x|.
\end{equation}
From  \eqref{dseq2} it follows that $|\exp(-tw)x|>r'$ if $t > r'+|x|$ and $|\exp(-tw)x|<r'$ if $t < r'-|x|$.
Hence by continuity and by Proposition \ref{lemma-convex}, for a fixed $w \in \Sph^{n-1}$ we can find a unique $t_w \in [r'-|x|, r'+|x|]$ such that
$|\exp(-t_ww)x|=r'$ and $|\exp(-tw)x|<r'$ if and only if $t< t_w$.
Similarly for a fixed $w \in \Sph^{n-1}$,  we can find a unique $s_w \in [r-|x|, r+|x|]$
with $|\exp(-s_ww)x|=r$ and $|\exp(-tw)x|<r$ if and only if $t< s_w$.

Therefore
\begin{eqnarray}\label{conv-formula-3}
\left| f \ast a_{r,r'}(x) \right| &=& \left| \frac 1{V_{r,r'}}\int_{S} f(y) \chi_{{\A_{r,r'}}}(y^{-1}x) dy \right|
 \\ &=& \left| \frac 1{V_{r,r'}}\int_{\Sph^{n-1}}\int_{{\R}^+}f(\exp tw)\chi_{\A_{r,r'}}(\exp(-tw)x) J(t)\ dt \  dw  \right| \nonumber \\
 &=& \left| \frac 1{V_{r,r'}}\int_{S^{n-1}}\int_{s_w}^{t_w}f(t) J(t)\ dt \ dw \right| \nonumber\\
 &\le&  \int_{S^{n-1}}\left|\frac {V_{s_w,t_w}}{V_{r,r'}}\right|\left|\frac{1}{V_{s_w,t_w}}\int_{s_w}^{t_w}f(t) J(t)\ dt \right|\ dw  \nonumber
\end{eqnarray}

If $r=r_j$ and $r'=r_j$ in \eqref{conv-formula-4}, then we get $s_w \in [r_j-|x|, r_j+|x|]\subset [r_j-\delta', r_j+\delta'], t_w \in [r'_j-|x|, r'_j+|x|]\subset [r'_j-\delta', r'_j+\delta']$. Hence by Lemma \ref{alpro1}, we get
\begin{equation}
\label{dseq4}
  d < t_w-s_w < d+\delta,
\end{equation}

and

\begin{eqnarray}\label{conv-formula-4}
\left|f \ast a_{r_j,r'_j}(x)\right| \le C \int_{S^{n-1}}\left|\frac{1}{V_{s_w,t_w}} \int_{s_w}^{t_w}f(t) J(t)\ dt \right|dw
\end{eqnarray}

From  \eqref{dseq5}, \eqref{dseq4} and \eqref{conv-formula-4}, it easily follows that
\[\displaystyle\lim_{j \to \infty} f \ast \mu_j(x) = 0.\]

\end{proof}

\begin{proof}[Completion of proof of Theorem \ref{annth1}]
Let $\{h_i\}_{i \in \N}$ be a sequence of continuous functions
 converging uniformly to $h$ over compact sets. Then we have the following observations.
\begin{itemize}
 \item[(a)]For any fixed $x \in S$,  $\ell_x h_i \to \ell_xh$ as $i \to \infty$ uniformly over compact
 sets.
 \item[(b)]   $R(h_i) \to R( h)$  pointwise as $i \to \infty$ .
\end{itemize}

Fix a point $x \in S$. By the hypothesis and observations (a), (b) we have
\[\ell_x(f \ast a_{r,r'}) \to \ell_xg\] uniformly on compact sets as $r \to \infty$ and
\[R(\ell_x(f \ast a_{r,r'})) \to  R(\ell_xg),\] pointwise as  $r \to \infty$. Since  $R(\ell_xf) \ast a_{r,r'}= R(\ell_x(f \ast a_{r,r'}))$, we have
\[R(\ell_xf) \ast a_{r,r'} \to  R(\ell_xg),\] pointwise as $r \to \infty$.
In particular $R(\ell_xf) \ast a_{r,r'}(e) \to  R(\ell_xg)(e)$.
By Lemma \ref{basic-lemma-3} (and using its notation), there exists a neighbourhood $N_e$ of $e$, such that for all $y \in N_e$,
\[\displaystyle\lim_{j \to \infty}R(\ell_xf) \ast \mu_j(y)= R(\ell_xg)(e)\varphi_\lambda(y).\]
Hence $R(\ell_xg)(y)= R(\ell_xg)(e)\varphi_\lambda(y)$ for all $y \in N_e$.  But as
$\mathscr M_{|y|}g(x)= R(\ell_xg)(y)$, we have $\mathscr M_{|y|}g(x)=g(x)\varphi_\lambda(y)$ for all $y \in N_e$. Proposition \ref{character-eigen} now asserts that
$\Delta g= -(\lambda^2+\rho^2)g$.
\end{proof}

\section{Proof of Theorem \ref{weth1}}
The proof of this theorem is essentially a much simpler version of the proof of Theorem \ref{annth1}. Nonetheless, for the sake of completeness we give here a quick sketch.   

The following  is an analogue of Lemma \ref{alpro1},  proved for the rank one symmetric spaces in \cite{NRS1}. Since it only uses some properties and estimates of the   Jacobi functions, the statement and the proof is valid for the Damek--Ricci spaces. 
\begin{lemma}\label{lem-ball-1}
Fix a $\lambda\in \C$. Then there exists a sequence of positive real numbers $\{r_n\}\uparrow \infty$ and a $\delta>0$ such that for any $n\in \N$, and any $r,s \in [r_n-\delta, r_n+\delta]$, $|V^\lambda_r|/|V^\lambda_s|\le C$ for some constant $C>0$. 
\end{lemma}
The next lemma is the fundamental step  towards the proof  and is a variant of Lemma  \ref{basic-lemma-3}. 

\begin{lemma}
\label{lem-ball-2}
Fix a $\lambda\in \C$. Let $f$  be a radial continuous function on $S$ such that $f\ast m_r^\lambda(e)\to 0$ as $r\to \infty$. Then there exists a $\delta>0$ and a sequence of of positive real numbers $\{r_n\}\uparrow \infty$ such that $f\ast m_{r_n}(x)\to 0$ as $n\to \infty$ for all $x\in S$ with $|x|<\delta$.
\end{lemma}
\begin{proof}
By triangle inequality we have
\[t-|x|\le |\exp(-tw)x|\le t+|x|\] for any $t\ge 0, w\in \Sph^{n-1}$ and $x\in S$.
Therefore, if $t>r+|x|$ then $|\exp(-tw)x|>r$ and if $t<r-|x|$ then $|\exp(-tw)x|<r$.
Hence by continuity of the function $t\mapsto |\exp(-tw)x|$ we have depending on $w$, $t_w\in [r-|x|, r+|x|]$ such that $|\exp (-t_w w) x|=r$.
From this and Proposition \ref{lemma-convex} we conclude that
$|\exp (-t w) x|<r$ if and only if $0<t<t_w$.
This, through the steps analogous to \eqref{conv-formula-3} leads to      
\[|f\ast m_r(x)|\le \int_{\Sph^{n-1}} \frac{|V_{t_w}|} {|V_{r}|}\ \ \left| 
\frac 1{V_{t_w}}\int_0^{t_w} f(t) J(t) dt\right| dw,\ \] where
$t_w\in [r-|x|, r+|x|]$ and $|\exp (-t_w w) x|=r$.

We take the sequence $\{r_n\}$ and $\delta>0$, prescribed by Lemma \ref{lem-ball-1}. Then for $|x|<\delta$,   
\[|f\ast m_{r_n}(x)|\le C \int_{\Sph^{n-1}} \ \ \left| 
\frac 1{V_{t_w}}\int_0^{t_w} f(t) J(t) dt\right| dw.\ \] 
This implies by the hypothesis that
\[\lim_{n\to \infty} f\ast m_{r_n}(x)=0.\]
\end{proof}
This lemma leads to a proof of  Theorem \ref{weth1}, following the argument given in the completion of proof of Theorem \ref{annth1}.
\section{Examples, counterexamples and concluding remarks}
\label{Examples and counter-examples}

In this concluding section of the paper we shall:

 (1) present some  simple examples of continuous functions $f$ on $S$ and examine the  asymptotic behaviour  $f\ast \sigma^\lambda_r$,

 (2) construct a counter example to show that  the condition  $r \to \infty$  in the hypothesis of the results obtained, cannot be replaced by ``$r$ approaches to $\infty$ through a sequence'',

 (3) discuss the reason for discarding an apparently natural formulation of the results and

(4) discuss an open question.

To keep the discussion simple, we shall restrict only to the sphere-averages. Similar ideas will lead to  examples and counterexamples for ball and shell averages also.

We need a preparatory lemma. 
 \begin{lemma} \label{various-cases-lemma}
 Let $\lambda, \mu \in \C$ with $\lambda \ne \mu$.
 Then we have the following conclusions. Below by $t\to \infty$ we mean $t\to \infty$ through  the set $\{t>0 \mid \varphi_\lambda(t)\neq 0\}$.
 \begin{itemize}
       \item [{\em (a)}]If $|\Im \lambda| > |\Im \mu|$,
       then $ \frac{\varphi_{\mu}(t )}{\varphi_{\lambda}(t )} \rightarrow 0$ as $t \rightarrow \infty$.
       \item [{\em (b)}]If $|\Im \lambda| < |\Im \mu|$,
       then $| \frac{\varphi_{\mu}(t )}{\varphi_{\lambda}(t )}|$  diverges to $\infty$ as $t \rightarrow \infty$.
 \item [{\em (c)}] If $|\Im \lambda| = |\Im \mu|$ and $\lambda\neq 0, \mu \neq 0$, then
        $ \frac{\varphi_{\mu}(t )}{\varphi_{\lambda}(t )}$ oscillates as $t \rightarrow \infty$.
    \item[{\em (d)}]If $\lambda=0$ and $\mu \in \R^\times$, then
        $ \frac{\varphi_{\mu}(t )}{\varphi_{\lambda}(t )} \rightarrow  0$ as $t \rightarrow \infty$.
     \item[{\em (e)}]If $\mu=0$ and $\lambda \in \R^\times$, then
        $ |\frac{\varphi_{\mu}(t )}{\varphi_{\lambda}(t )}|$ diverges  to $\infty$ as $t \rightarrow \infty$.
 \end{itemize}
\end{lemma}

\begin{proof}
Without loss of generality we shall assume that $\Im \lambda \le 0, \Im \mu \le 0$. Using \eqref{varphiesti1} and \eqref{phioesti}, it is clear that if $\Im \mu >\Im  \lambda$, then  \[\lim_{t\to \infty}\frac {|\varphi_{\mu}(t)|}{ |\varphi_{\lambda}(t)|}= 0\] and if $\Im \mu < \Im  \lambda$, then \[\lim_{t \rightarrow \infty} \frac{|\varphi_{\mu}(t)|}{|\varphi_{\lambda}(t)|}= \infty.\] This proves  (a)  and (b).

For (c) we have the following two cases.
\begin{itemize}
 \item [Case (i)] Let $\Im \mu = \Im  \lambda <0$. Then,
\begin{equation}
\label{baeq9}
 \displaystyle \lim_{t \rightarrow \infty} \frac {\varphi_{\mu}(t)}{\varphi_{\lambda}(t)}
  = \displaystyle \lim_{t \rightarrow \infty}  \frac {e^{(-i\mu+\rho)(t)}\varphi_{\mu}(t)}  { e^{(-i\lambda +\rho)t}\varphi_{\lambda}(t)}
\frac   { e^{(-i\lambda +\rho)t}}  { e^{(-i\mu +\rho)t}}
= \frac {\hc(\mu)} {\hc(\lambda)} \displaystyle \lim_{t\rightarrow \infty } { e^{-i(\lambda -\mu)t}}.
\end{equation}
Since $\lambda-\mu \in \R$,  $\lim_{t \rightarrow \infty} {\varphi_{\mu}(t)}/{\varphi_{\lambda}(t)}$ is oscillatory.

\vspace{0.4 in}

 \item[Case (ii)] Let $\Im \mu = \Im  \lambda =0$.
 Owing to \eqref{speesti1} we get
\begin{equation}
\label{baeq10}
\lim_{t \to \infty} \frac {\varphi_{\mu}(t)}{\varphi_{\lambda}(t)}= \frac{C_\mu}{C_\lambda}\lim_{t \to \infty}\frac{ \cos (\mu t +\theta_\mu ) +\tilde{E}(\mu,t)}{  \cos (\lambda t +\theta_\lambda )+\tilde{E}(\lambda,t)}.
\end{equation}
where $\tilde{E}(\lambda,t) = E(\lambda,t)/C_\lambda$ and $\tilde{E}(\mu,t) = E(\mu,t)/C_\mu$.
We shall show that \[\lim_{t \to \infty}\frac{ \cos (\mu t +\theta_\mu ) +\tilde{E}(\mu,t)}{  \cos (\lambda t +\theta_\lambda )+\tilde{E}(\lambda,t)}\] is oscillatory, dividing  it in two subcases.

We assume first that $\xi = \mu /\lambda$ is irrational.  Then by Kronecker's approximation theorem  $\{2n \pi \xi \,(\text{mod }\, 2\pi)  \mid n\in \N \}$ is dense in $[0,2\pi]$. We construct a sequence $\{t_n\}_{n \in \N}$ where
\[t_n = 2n \pi/\lambda-\theta_\lambda/\lambda, n \in \N.\]
 Then
 \[\lim_{n \to \infty}\cos(\lambda t_n +\theta_\lambda)+\tilde{E}(\lambda,t_n) =1.\]
Since $\mu t_n +\theta_\mu =2n\pi\xi-\theta_\lambda \xi+\theta_\mu$, it follows that  \[\{\mu t_n +\theta_\mu \, (\text{mod } 2\pi) \mid n \in \N \}\] is also dense in $[0, 2 \pi]$.
Therefore for any $L \in [-1, 1]$, there is a subsequence $\{t_{n_k}\}_{k\in \N}$ of $\{t_n\}_{n \in \N}$ such that
 \[\lim_{k \to \infty}\cos(\mu t_{n_k}+\theta_{\mu})=L.\]
 Hence,
 \[\lim_{k \to \infty} \frac{\cos(\mu t_{n_k} +\theta_{\mu})+\tilde{E}(\mu,t_{n_k})}{\cos(\lambda t_{n_k}+\theta_\lambda)+\tilde{E}(\lambda,t_{n_k})}=L.\]

 If $\mu/\lambda$ is rational we assume that $\mu/\lambda= m/n$, for $m, n\in \Z$. We choose a $\xi\in \R$
 such that $\cos(\lambda \xi +\theta_{\lambda}) \ne 0$ and construct a sequence $\{t_k\}$ with $ t_k= \xi +\frac {2 n \pi k}{\lambda}, k\in \N$. Then,
 \[\lim_{k \to \infty} \frac{\cos(\mu t_{k} +\theta_{\mu})+\tilde{E}(\mu,t_{k})}{\cos(\lambda t_{k}+\theta_\lambda)+\tilde{E}(\lambda,t_{k})}= \frac{\cos(\mu \xi +\theta_{\mu})}{\cos(\lambda \xi+\theta_\lambda)},\] which is oscillatory as  $\xi \in \R$ is arbitrary and $\mu \ne \lambda$.

\end{itemize}
This completes the proof of $(c)$. Using \eqref{speesti1} and \eqref{phioesti}, (d) and (e) also follows easily.
\end{proof}
\vspace{.2in}

\noindent {\bf(1)} 
An immediate consequence of this lemma is the following.
\begin{proposition}\label{various-cases1}
 Let $\lambda, \mu \in \C$ with $\lambda \ne \mu$ and $f = \varphi_\lambda+\varphi_\mu$.
 Then we have the following conclusions, where by $r\to \infty$ we mean
 $r\to \infty$ through  the set $\{r>0 \mid \varphi_\lambda(r)\neq 0\}$.
  \begin{itemize}
       \item [\em {(a)}]If $|\Im \lambda| > |\Im \mu|$,
       then $ f \ast \sigma^\lambda_r (x)\rightarrow \varphi_\lambda(x)$  for all $x \in S$ as $r \rightarrow \infty$.
       \item [\em {(b)}]If $|\Im \lambda| < |\Im \mu|$,
       then $f \ast  \sigma^\lambda_r(x)$  diverges for all $x \in S$ as $r \rightarrow \infty$.
 \item [\em {(c)}]If $|\Im \lambda| = |\Im \mu|$ and $\lambda \ne 0, \mu \ne 0$, then
        $ f \ast  \sigma^\lambda_r(x)$ oscillates for all $x \in S$ as $r \rightarrow \infty$.
    \item[\em {(d)}]If $\lambda=0$ and $\mu \in \R^\times$, then
        $ f \ast  \sigma^\lambda_r(x) \rightarrow \varphi_\lambda(x)$  for all $x \in S$  as $r \rightarrow \infty$.
     \item[\em {(e)}]If  $\mu=0$ and $\lambda \in \R^\times$, then
        $ f \ast  \sigma^\lambda_r(x)$ diverges for all $x \in S$ as $r \rightarrow \infty$.
 \end{itemize}
\end{proposition}

\begin{proof}
Without loss of generality we shall assume that $\Im \lambda \le 0, \Im \mu \le 0$.
Since  \[f \ast \sigma^\lambda_r (x)= \varphi_\lambda(x)+ \frac{\varphi_\mu(r)}{\varphi_\lambda(r)}\varphi_\mu(x)\] for all $x\in S$, in view of Lemma \ref{various-cases-lemma} the proposition follows.
\end{proof}
\vspace{.2in}

\noindent {\bf(2)} 
Theorem \ref{weth5} is not true if the radius $r$ approaches $\infty$ via an arbitrary sequence. Precisely,  for any $\alpha=-(\lambda^2+\rho^2)\in \C$ there exist continuous  functions $f, g$ and a sequence $r_n\uparrow \infty$ such that $f\ast \sigma_{r_n}^\lambda(x)\to g(x)$ uniformly on compact sets, but $g$ is not an eigenfunction with eigenvalue $\alpha$. Here is an example.
We fix $\alpha=-(\lambda^2+\rho^2)\in \C$. First we assume that $\Im \lambda<0$. We take  $\mu \in \C$ such that  $\Im \lambda = \Im \mu$ and $\Re \mu <\Re \lambda$.
Let $f(x) = \varphi_\mu(x)$.  Since $\lambda -\mu>0$,  the sequence  $r_n=2n\pi/(\lambda-\mu)$  of positive real numbers diverges to $\infty$ and $e^{-i(\lambda -\mu)r_n}=1$. 
Hence by  \eqref{baeq9}  
\[\lim_{n \to \infty} f \ast \sigma^\lambda_{r_n}(x) = \lim_{n \to \infty} \frac {\varphi_\mu(r_n)}{\varphi_{\lambda}(r_n)} \varphi_\mu(x)=
 \frac {\hc(\mu)} {\hc(\lambda)} \varphi_\mu(x).\] Thus $f \ast \sigma^\lambda_{r_n}$ converges to  $g= \frac {\hc(\mu)} {\hc(\lambda)} \varphi_\mu$ and $\Delta g \ne -(\lambda^2+\rho^2)g$.

Similar construction works for $\lambda\in \R$, $\lambda\neq 0$. Precisely we take again $f=\varphi_\lambda+\varphi_\mu$ where $\mu\in \R^\times$ and $\lambda\neq \mu$. 
By  Lemma \ref{various-cases-lemma}(c), we  can obtain a real number $L \ne 0$ and
 sequence  $\{r_n\}$ with $r_n \uparrow \infty$  and
 \[\lim _{n \to \infty}\frac {\varphi_{\mu}( {r_n} )}{\varphi_{\lambda}({r_n} )} =L.\]
 Hence we get
 \[\lim_{n \to \infty}f \ast \sigma^\lambda_{r_n}(x) =
 L \varphi_\mu(x),\] and the limit is not an eigenfunction of $\Delta$ with the prescribed eigenvalue $-(\lambda^2+\rho^2)$.
\vspace{.2in}

\noindent {\bf (3)} 
The generalized sphere mean value property \eqref{MVP-sphere} seems to suggest the following:
If a continuous function $f$  on $S$ satisfies
\[\lim_{r \to \infty } |\mathscr M_rf-\varphi_\lambda(r)f|=0\]
uniformly on compact set of $S$, then  $\Delta f= -(\lambda^2+\rho^2) f$. 
It is indeed not true, as can be illustrated through the following counterexample. Let $\lambda, \mu \in \C$ be such that $|\Im \lambda| < \rho, |\Im \mu| <\rho$ and $\lambda \ne \pm \mu$. We take $f = \varphi_\lambda +\varphi_\mu$. Then since $\mathscr M_r f= \varphi_\lambda(r) \varphi_\lambda +\varphi_\mu(r)\varphi_\mu$, $\varphi_\lambda(r) \to 0$ and $\varphi_\mu(r)\to 0$
as $r\to \infty$ (see \eqref{varphiesti1}), it  follows that the hypothesis is satisfied by $f$, but $f$ is clearly not an eigenfunction of $\Delta$.
\vspace{.2in}

\noindent {\bf (4)} In Theorem \ref{annth1} the radii $(r', r)$ of the annuli  belong to a strip in the first quadrant of the plane. Can this strip  be replaced by a curve, in particular, can we take radii of the annuli from a straight line, e.g. $(r+d, r)$ for a fixed $d>0$? We do not know the answer and hope one of  our readers will explore this question.

\end{document}